\documentclass{amsart}
\usepackage{graphicx}
\usepackage{subcaption}
\usepackage{amsmath, amsfonts, amssymb, amsthm, mathtools}
\usepackage{microtype}
\usepackage{hyperref}
\usepackage{booktabs}
\usepackage{siunitx}
\usepackage{tikz, pgfplots}
\pgfplotsset{compat=newest}
\usetikzlibrary{spy} 
\usetikzlibrary{backgrounds}
\pgfplotsset{
    compat=1.15,
    mesh plot style/.style={
        xtick={31, 63, 127, 255, 511},
        xticklabels={31,63,127,255, 511},
        xlabel={Mesh size ($N$)},
        width=1.1\textwidth,
        legend style={font=\tiny, row sep=-2pt, draw=none, fill=white, inner sep=1pt},
        legend cell align=left,
        grid=major,
    },
    mesh plot style 2/.style={
        xlabel={Running time ($s$)},
        width=1.1\textwidth,
        legend style={font=\tiny, row sep=-2pt, draw=none, fill=white, inner sep=1pt},
        legend cell align=left,
        grid=major,
    },
    M2 style/.style = { green!80!black, thick, mark size= 2.5pt, mark=| },
    bellman style/.style = { blue, thick, mark=o, mark size=2.5pt },
    M1 style/.style = { orange!80!black, thick, mark size=2.5pt, mark=x },
}

\numberwithin{equation}{section}

\DeclareMathOperator{\Leb}{Leb}
\DeclareMathOperator{\tr}{tr}
\DeclareMathOperator{\MA}{MA}
\DeclareMathOperator{\cof}{cof}
\theoremstyle{plain}
\newtheorem{theorem}{Theorem}[section]

\newtheorem{corollary}[theorem]{Corollary}

\theoremstyle{definition}

\newcommand{\weakstar}{\mathop{w^*\text{-}\lim}}

\theoremstyle{remark}
\newtheorem*{remark}{Remark}

\begin{document}

\title[A numerical method for the Monge--Ampère equation]{A fast Bellman algorithm for  the real Monge--Ampère equation}


\author{Aleksandra Le}
\email{ai\_my\_aleksandra.le@math.lth.se}
\address{Centre for Mathematical Sciences\\
  Lund University\\
  Box 118, SE-221 00 Lund, Sweden}
  
\author{Frank Wikström}
\email{frank.wikstrom@math.lth.se}
\address{Centre for Mathematical Sciences\\
  Lund University\\
  Box 118, SE-221 00 Lund, Sweden}

\begin{abstract}
In this paper, we introduce a new numerical algorithm for solving the Dirichlet problem for the real Monge--Ampère equation. The idea is to represent the non-linear Monge--Ampère operator as an infimum of a class of linear elliptic operators and use Bellman's principle to construct a numeric scheme for approximating the operator attaining this infimum.

Moreover, we prove convergence of the proposed algorithm (under suitable technical assumptions) and discuss its strengths and weaknesses. We also demonstrate the performance of the method on several examples with various degrees of regularity and degeneracy and compare the results to two existing methods. Our method runs considerably faster than the ones used for comparison, improving the running time by a factor of 3--10 for smooth, strictly convex examples, and by a factor of 20--100 or more for mildly degenerate examples.
\end{abstract}

\maketitle

\section{Introduction}

The Monge--Ampère equation is a fully non-linear partial differential equation,
originating from the study of surfaces with prescribed curvature which goes 
back to Monge in the late 18th and Ampère in the early 19th century.
Since then, variations of this equation have found a number of applications, 
for example in optimal transport~\cite{Philippis, Villani} and Riemannian geometry
(the local isometric embedding problem~\cite{Lin}). 
The equation is also of great theoretical interest and is often used as a model 
example of a fully non-linear PDE.  The complex version of the Monge--Ampère equation  plays a central role 
in pluripotential theory,  i.e.\ in the study of plurisubharmonic functions on 
domains in $\mathbb{C}^n$ and more recently on Kähler manifolds~\cite{Guedj-Zeriahi}.

In this paper, we introduce a novel numerical method for solving the Dirichlet 
problem for the real Monge--Ampère equation, based on \emph{Bellman's principle}: 
representing the non-linear operator as an infimum of elliptic linear partial
differential operators. We construct a numeric iterative scheme for approximating 
the operator attaining this infimum. 

Under suitable technical assumptions, as well as the assumption that our algorithm 
actually produces a sequence of strongly convex approximants, we prove convergence,
see Theorem~\ref{thm:convergence} for details.
Also, any convex approximants always majorises the solution to the original Monge--Ampère equation.

Numerical experiments illustrate that the method works well in settings where the 
solution is smooth and strictly convex. In these cases, it outperforms recently 
published methods by a factor of~3--10 in execution time. For mildly singular equations, 
the speed-up is even better: in some cases improving the running time by a factor of~100 or more.
In the fully degenerate case, when the right hand side of the equation vanishes on a 
non-empty open set, our proposed method does not converge in general.
We intend to return to the fully degenerate problem in a subsequent paper.

\section{Background}

The Dirichlet problem for the Monge--Ampère operator can be formulated as
\begin{equation}\label{eq:MA}
    \begin{cases}
        \det(D^2u(x)) = f(x), & x \in \Omega, \\
        u(x) = \phi(x) ,      & x \in \partial\Omega,
    \end{cases}
\end{equation}
where $\Omega \subset \mathbb{R}^n$ is a convex domain, $f \ge 0$ and the boundary 
values of  $u$ are suitably interpreted. Here, $D^2 u = (u_{jk})$ denotes the Hessian 
of~$u$. By restricting our attention to convex solutions of~\eqref{eq:MA},
the equation is (degenerate) elliptic, and we will only consider 
such solutions. If we 
impose some extra conditions on $\Omega$, $f$ and  $\phi$, we can guarantee that 
the equation has a unique solution. Note that the left-hand side, 
\emph{the Monge--Ampère operator}, $\MA(u) := \det(D^2u(x))$ makes perfect sense for~$C^2$ 
functions, and it can be interpreted pointwise also for functions in~$W_{\text{loc}}^{2,n}$. 
For functions of lower regularity, the Monge--Ampère operator 
must be viewed in  a generalized sense, and since the operator is fully non-linear, 
there is no direct  interpretation in the sense of distributions. However, convex 
functions are  automatically locally Lipschitz and it turns out that we can define 
the Monge--Ampère measure (in the sense of Aleksandrov) of any convex function~$u$, by
\[
    \mu_u(E) = \Leb(\partial u(E))
\]
for any Borel subset $E \subset \Omega$, where $\Leb$ denotes the Lebesgue measure,
$\partial u(E) = \cup_{x\in E} \partial u(x)$, and $\partial u(x)$ is the subgradient 
of $u$, i.e.\ the set of all 
slopes of the supporting hyperplanes to the graph of $u$ at $x$. If $u \in C^2$ 
and $D^2 u$ is positive semi-definite, then $\mu_u(E) = \det(D^2 u)$. We should mention, 
that in contrast to the real setting, it is impossible to define the complex Monge--Ampère
operator for every plurisubharmonic function. We refer to the monograph by Guedj and 
Zeriahi~\cite{Guedj-Zeriahi} for the theory of complex Monge--Ampère equations. 

There is a vast number of analytic results in the literature concerning existence and 
regularity of solutions  to~\eqref{eq:MA}. A complete overview is outside the 
scope of this paper, but for example, it follows from results by Aleksandrov~\cite{Alexandrov}
and Bakelman~\cite{Bakelman} from the 1950's that if~$\Omega$ is assumed to be a bounded 
strictly convex domain, $f$ is  a finite Borel measure on~$\Omega$, and 
$\phi \in C(\partial\Omega)$,  then~\eqref{eq:MA} has a unique convex solution~$u$ (in
the sense of Aleksandrov) with $u \in C(\bar\Omega)$.  

Furthermore, if $\Omega$ is strictly convex (or uniformly convex in the unbounded 
setting) with $C^{k+2, \alpha}$ boundary, $0 < \phi \in C^{k+2, \alpha}(\partial\Omega)$
and $f \in C^{k,\alpha}(\bar\Omega)$ for some integer $k \ge 2$ and some $0 < \alpha < 1$, 
then the Dirichlet problem~\eqref{eq:MA} has a unique convex solution
$u \in C^{k+2, \alpha}(\bar\Omega)$, see~\cite{CKNS, Ivockina, Krylov}. 

If we merely  assume that $\Omega$ is convex, not necessarily strictly convex, 
the situation is a little less straight-forward, and we need to assume that~$\phi$ 
is convex on  all line segments in $\partial\Omega$ to guarantee 
existence of solutions to~\eqref{eq:MA}. Similarly, if~$f$ is only 
assumed to be non-negative instead of strictly positive, the solutions
to the Dirichlet problem will generally be less smooth, even in strictly
convex domains. If $f$ has zeros, the equation is called \emph{degenerate}.
See for example B\l{}ocki~\cite{Blocki}  for results in this
direction, where $C^{1,1}$-regularity of the solution is shown under 
suitable technical assumptions. We will come back to this phenomenon in example~\ref{example:constant_MA}.

For an accessible and thorough overview of the Aleksandrov solutions as well as 
results on classical solutions including the ones mentioned, we refer to the recent 
monographs by Figalli~\cite{Figalli} and Le~\cite{Le}.

There have been a number of suggested numerical methods for solving the Dirichlet
problem for the Monge--Ampère operator. As far as we are aware, the first approach 
was due to Oliker and Prussner (1988)~\cite{Oliker} who proposed a direct discretization 
of the Aleksandrov interpretation of the Monge--Ampère operator. Since then, a number of
other approaches to solving Monge--Ampère equations have been tried: finite difference
discretization, often using wide stencils to obtain stable and consistent 
methods~\cite{Benamou, Benamou2, Froese, Froese2, Oberman}, variational
methods~\cite{Belgacem:opt, Feng}, fixed-point methods~\cite{Benamou, Belgacem} as well as 
finite elements~\cite{Feng:FEM, Lakkis} only to mention a few.

After the preparation of an earlier draft of this paper, we learned that a similar 
method as ours (for periodic boundary values) has been previously proposed by Loeper and 
Rapetti~\cite{Loeper}. Their approach is somewhat different than ours, and derive their 
method as a Newton type iteration for the non-linear differential operator.

\section{Bellman's principle}

Our method relies on the following linear algebra result, sometimes called 
\emph{Bellman's principle}. In the setting of complex Monge--Ampère equations, 
this idea was introduced by Gaveau~\cite{Gaveau} and has later found several 
theoretical applications, both in the real and the complex setting. 
We will formulate the result in the real 
setting and for convenience of the reader, we provide a proof of Bellman's 
principle.

\begin{theorem}\label{thm:bellman}
    Let $M$ be a positive semi-definite symmetric $n \times n$ real matrix. Then
    \begin{equation}\label{eq:bellman}
        (\det M)^{1/n} = \frac{1}{n} \inf_{B \in \mathcal{B}} \tr(BM),
    \end{equation}
    where $\tr(\cdot)$ denotes the trace of a matrix, and $\mathcal{B}$ denotes the set of positive definite symmetric $n \times n$-matrices
    with determinant~$1$. Also, if $M$ is positive definite, the infimum is attained by
    some matrix~$B$. 
\end{theorem}

\begin{proof}
   First, we show that
    \begin{equation*}
        (\det M)^{1/n}\leq \frac{1}{n}\inf_{B\in \mathcal{B}}\tr(BM).
    \end{equation*}
    To this end, we observe that for any symmetric and positive definite matrix $B$, we have 
    \begin{equation*}
    \begin{split}
        \det(BM-\lambda I) 
        &= \det(B^{1/2}B^{1/2}MB^{1/2}B^{-1/2}-\lambda I) \\ 
        &= \det(B^{1/2})\det(B^{1/2}MB^{1/2}-\lambda I)\det(B^{-1/2}) \\
        &= \det(B^{1/2}MB^{1/2}-\lambda I).
    \end{split}
    \end{equation*}
    Since $B^{1/2}MB^{1/2}$ is symmetric and positive semi-definite, it can be diagonalised by a unitary matrix $Q$ and the eigenvalues of $BM$, $\{\lambda_{k}\}_{1\leq k\leq n}$, are non-negative. 
    Moreover, for $B$ Hermitian with $\det(B)=1$ we have 
    \begin{equation*}
    \begin{split}
        \det(M) = \det (BM) &= \prod_{1\leq k\leq n} \lambda_{k}\leq \Big(\frac{1}{n}\sum_{1\leq k\leq n}\lambda_{k}\Big)^{n} \\
        &= \Big(\frac{1}{n}\tr(Q^{*}B^{1/2}MB^{1/2}Q)\Big)^{n} \\ &
         = \Big(\frac{1}{n}\tr(B^{1/2}QQ^{*}B^{1/2}M)\Big)^{n}=\Big(\frac{1}{n}\tr(BM)\Big)^{n},
    \end{split}
    \end{equation*}
    using the inequality between the arithmetic and geometric mean.
    Next, we show that the infimum is attained for positive definite matrices. If $M$ is a positive definite matrix then $B\coloneqq (\det M)^{1/n}M^{-1}$ is a well-defined, positive definite matrix with $\det B=\det M \cdot \det M^{-1}=1$ which satisfies the statement. Indeed,
    \begin{equation*}
    \tr{BM}=\tr{((\det M)^{1/n}M^{-1}M)}=(\det M)^{1/n}n.
\end{equation*}
    If $M$ is diagonal and positive semi-definite, let $p$ be the rank of $M$. Clearly, $p<n$. Without loss of generality, we can assume that $\lambda_{k}=0$ for $k\in \{ p+1, \dots, n \}$. In this case, the equality can be obtained by choosing any constant $c>0$ and constructing a diagonal matrix $B$ with $b_{kk}\coloneqq c$ for $k\in \{1,\dots,p\}  $ and $b_{kk}\coloneqq c^{-p/(n-p)}$ for $k\in \{p+1, \dots n\} $ so that $\det(B)=1$. Then, we have
    \begin{equation*}
        \begin{split}
    \tr(BM) &=\sum_{k=1}^{p}b_{kk}\lambda_{k}=c\cdot \tr(M) \implies \\
    &\qquad \inf_{B\in \mathcal{B}}\tr{(BM)}\leq \inf_{c>0}c\cdot \tr(M)=0=n\det(M)^{1/n}.
        \end{split}
    \end{equation*}
    Combining this with $n \det(M)^{1/n}\leq \inf_{B\in \mathcal{B}} \tr(BM)$ derived at the beginning, we obtain Equation \eqref{thm:bellman}.
    For a positive semi-definite matrix $M$ that is not diagonal, it suffices to notice that there exists $T$ such that $M=T^{*}\Lambda T$, where $\Lambda$ is a diagonal matrix with the eigenvalues of $M$ on the diagonal. By the above, we have
    \begin{equation*}
    \begin{split}
        (\det M)^{1/n} & = (\det \Lambda)^{1/n}= \frac{1}{n}\inf_{B\in \mathcal{B}}\tr(B\Lambda)= \frac{1}{n}\inf_{B\in \mathcal{B}}\tr(BTMT^{*})\\ & = 
        \frac{1}{n}\inf_{B\in \mathcal{B}}\tr(T^{*}BTM) = \frac{1}{n}\inf_{B' \in \mathcal{B}}\tr(B'M).\qedhere
    \end{split}
    \end{equation*}
\end{proof}
Using Bellman's principle, we can describe the Monge--Ampère operator (of a
sufficiently smooth  function) as an infimum over a class of elliptic partial 
differential operators. More precisely:

\begin{corollary} \label{cor: Bellmans principle}
    Let $\Omega$ be a domain in $\mathbb{R}^n$ and assume that 
    $u \in C^2(\Omega)$ is convex. Then
    \begin{equation*}
        (\MA(u)(x))^{1/n} = \frac{1}{n} \inf_{B \in \mathcal{B}} \tr\big(B\,D^2u(x))\big).
    \end{equation*}
    In particular, if $u$ is strongly convex, the infimum is attained by the $\mathcal{B}$-valued function
    \[ B(x)=(\det D^{2}u)^{1/n}(D^{2}u)^{-1}(x). \]
\end{corollary}

Note that, for every choice of $\mathcal{B}$-valued function~$B$, the 
operator 
\[ u \mapsto \tr\big(B(x)D^2u(x)\big) \]
is a linear elliptic partial differential operator, and if we somehow know 
the associated $B(x)$ corresponding to the solution $u$ of the Dirichlet problem~\eqref{eq:MA}, then we can recover $u$ by solving the \emph{linear}
elliptic Dirichlet problem
\begin{equation}\label{eq:linear_eq}
    \begin{cases}
        \tr(B(x)\,D^2u(x)) = n\,\big(f(x)\big)^{1/n}, & x \in \Omega, \\
        u(x) = \phi(x) ,                              & x \in \partial\Omega,
    \end{cases}
\end{equation}
for which there are highly efficient numerical algorithms. 
% Let us mention that with this approach we have the following monotonicity result.

% \begin{proposition}\label{prop:monotone}
%     Let $\Omega \subset \mathbb{R}^n$ be a convex domain with $C^{2,\alpha}$ 
%     boundary,  $\phi \in C^{2,\alpha}(\partial\Omega)$, $f^{1/n} \in C^{0,\alpha}(\Omega)$
%     and let $B \in C^{0,\alpha}(\Omega)$ be a $\mathcal{B}$-valued 
%     function on $\Omega$. Assume that the solution to the above Dirichlet problem (Eq. \eqref{eq:linear_eq})
%     is convex. Then $u \ge u_{\text{MA}}$ where $u_{\text{MA}}$ denotes the solution 
%     to the corresponding Monge--Ampère equation~\eqref{eq:MA}.
% \end{proposition}

% \begin{proof}
%     By standard regularity theory for elliptic PDEs, $u \in C^{2,\alpha}$, 
%     so by Corollary~\ref{cor: Bellmans principle}, $\MA(u) \le f$. Hence, by the comparison 
%     principle for the real Monge--Ampère operator (see for example~\cite[Theorem~3.21]{Le}),
%     it follows that $u \ge u_{\text{MA}}$.
% \end{proof}

Unfortunately, it seems like there are no known effective characterizations of when 
the solution to Equation~\eqref{eq:linear_eq} is convex, even in the case when $B \equiv I$, i.e.
for Poisson's equation. 
% Our numerical experiments suggest that our proposed algorithm (as described in Section \ref{sec:algorithm}) produces a decreasing sequence of approximants converging to the solution of~\eqref{eq:MA},
% at least as long as the approximants are strictly convex. 
Under suitable technical conditions and the assumption that our method
produces a sequence of strictly convex approximants, we can show convergence, see Theorem~\ref{thm:convergence}.

We should point out that there is no proof of convergence for most methods that have been proposed for the Monge--Ampère equation. In particular,
this is the case for the M1 method we compare our method to in Section~\ref{sec:examples}. The convergence (under suitable conditions) of method M2, also discussed in Section~\ref{sec:examples}, has been recently proved in~\cite[Theorem 4.1]{Chen}, at least in the case when the boundary data is equal to zero. Notable methods with a proof of convergence are the direct discretization of the Aleksandrov construction proposed in~\cite{Oliker} and the wide stencil finite difference discretizations~\cite{Oberman} which are also known to converge (to a viscosity solution).

\section{Proof of convergence}

The constructive proof of Bellman's principle, i.e.\ Theorem~\ref{thm:bellman}, suggests a method of constructing a sequence of linear Dirichlet problems of the form~ \eqref{eq:linear_eq} converging to the Dirichlet problem for the Monge--Ampère equation~\eqref{eq:MA}. The result is the theorem below.
\begin{theorem} \label{thm:convergence}
    Let $\alpha\in (0,1)$ and let $\Omega\subset \mathbb{R}^{n}$ be a strictly convex, bounded, $C^{2,\alpha}$ domain, $\phi\in C^{2,\alpha}(\bar{\Omega})$ and $f\in C^{0,\alpha}(\bar{\Omega})$ with strictly positive values on $\bar{\Omega}$. Assume that there exist solutions $\{u_{k}\}_{k\geq 0}\subset C^{2,\alpha}(\bar{\Omega})$ to the following linear Dirichlet problems
    \begin{equation} \label{eq: linear system of pdes}
    \begin{cases}
        \tr(B_{k}(x)D^{2}u_{k}(x))=nf^{1/n}(x), & x \in \Omega\\
        u_{k}(x)=\phi(x), & x \in \partial \Omega \\
        u_{k}(x) \text{ is strongly convex}.
    \end{cases}
\end{equation}
where $B_{k}=(\det{D^{2}u_{k-1}})^{1/n} (D^{2}u_{k-1})^{-1}=\cof D^{2}u_{k-1}$ for $k\geq 1$ and $B_{0}$ be a $\mathcal{B}$-valued function with coefficients $(B_{0})_{ij}\in C^{0,\alpha}(\bar{\Omega})$. Moreover, denote by $\lambda_{\text{min}}^{k}$ the smallest eigenvalue of $D^{2}u_{k}$ and assume that $\{\lambda_{\min}^{k}\}_{k\geq 0} $ is uniformly bounded away from zero, i.e.\ there exists $\epsilon>0$ such that $\lambda_{\text{min}}^{k} \geq \epsilon$ for every $k\geq 0$. Then, the sequence of solutions $\{u_{k}\}_{k\geq 0}$ converges uniformly to a convex function $\tilde u \in C^{1,1}(\bar\Omega)$ solving the Monge--Ampère equation~\eqref{eq:MA}.
\end{theorem}
Before proving the above theorem, let us first discuss the requirement of the existence of solutions to \eqref{eq: linear system of pdes} and their strong convexity.
Under the regularity assumptions stated above on the domain $\Omega$, the functions $f$ and $\phi$,  and the coefficients $(B_{0})_{ij}$, standard theory of linear elliptic PDEs (see Gilbarg and Trudinger~\cite[Theorem 6.14]{Gilbarg}) provides existence and uniqueness of a solution $u_{0}\in C^{2,\alpha}(\bar{\Omega})$.  

Next, let us investigate solvability and well-definiteness of the next Dirichlet problem in the sequence. We would like to generalise this argument for any $k\geq 0$. To this end, assume the existence of the solution $u_{k}\in C^{2,\alpha}(\bar{\Omega})$.

Observe that $B_{k+1}\coloneqq (\det D^{2}u_{k})^{1/n}\,(D^{2}u_{k})^{-1}$ is not necessarily well-defined, nor is the next differential operator $L_{k+1}(v)\coloneqq \tr(B_{k+1}D^{2}v)$ strictly elliptic. Consequently, the solution to the corresponding Dirichlet problem is not guaranteed.

However, strong convexity of $u_{k}$ renders the next differential operator $L_{k+1}$ well-defined and strictly elliptic. Indeed, first note that from $\epsilon \leq \lambda_{\text{min}}^{k}$ it follows that $D^{2}u_{k}$ is positive definite and hence invertible.
Recall that $B_{k+1}\coloneqq (\det D^{2}u_{k})^{1/n}(D^{2}u_{k})^{-1}$, by the above $B_{k+1}$ is well-defined and positive definite with $\det B_{k+1}=1$. We can use the Corollary to Bellman's principle 
    (Corollary~\ref{cor: Bellmans principle}) to conclude that  
\begin{equation*} 
    n\epsilon \leq n(\det D^{2}u_{k})^{1/n}=\tr{B_{k+1}D^{2}u_{k}}\leq \tr{B_{k}D^{2}u_{k}}=nf^{1/n} \text{ on } \bar{\Omega}.
\end{equation*}
Thus, 
\begin{align} 
    \epsilon^{n} \leq \det D^{2}u_{k}\leq f \text{ on } \bar{\Omega}. \label{eq: bounds on nu_k}
\end{align}
Moreover, existence of a uniform lower bound of the eigenvalues $\lambda^{k}_{min}$ implies existence of a uniform upper bound for the largest eigenvalues denoted by $\lambda^{k}_{max}$. Indeed, Eq. \eqref{eq: bounds on nu_k} yields 
\begin{equation*}
    \epsilon  \leq \lambda^{k}_{min}\leq \lambda^{k}_{max}\leq f/\epsilon^{n-1} , \qquad\text{for $k\geq 0$.} \label{eq: eigenvalue bounds}
\end{equation*} 
By diagonalising $B_{k+1}$ we have that the smallest and the largest eigenvalues are uniformly bounded as follows
\begin{equation*} \label{eq: bounds for eigenvalues of B_k}
\begin{split}
    \gamma_{min}^{k}& =(\det D^{2}u_{k})^{1/n}/\lambda_{max}^{k}\geq (\epsilon^{n})^{1/n}/(f/\epsilon^{n-1})\geq \epsilon/\|f\|_{C(\bar{\Omega})} \text{ and }\\
     \gamma_{max}^{k}& =(\det D^{2}u_{k})^{1/n}/\lambda_{min}^{k}\leq \|f^{1/n}\|_{C(\bar{\Omega})}/\epsilon.
\end{split}
\end{equation*}
From the lower bound of $\gamma_{min}^{k}$ follows strict ellipticity of $L_{k+1}$. Furthermore, the regularity of the solution, i.e. $u_{k}\in C^{2,\alpha}(\bar{\Omega})$ results in $(B_{k+1})_{ij}\in C^{0,\alpha}(\bar{\Omega})$ as $(B_{k+1})_{ij}$ are obtained by adding, subtracting and multiplying terms  $\partial_{ij}^{2}u_{k}$.

Thus $L_{k+1}$ is strictly elliptic with the necessary regularity of coefficients for the corresponding Dirichlet problem to have a unique solution $u_{k+1}\in C^{2,\alpha}(\bar{\Omega})$ (by another application of Theorem 6.14 in Gilbarg and Trudinger).

\begin{proof}
We begin with showing monotonicity and hence convergence of the sequence of solutions 
$\{u_{k}\}_{k\geq 0}$. 

Let us define $\nu_{k}\coloneqq \det D^{2}u_{k}(x)$ and $\eta_k \coloneqq u_{k+1}-u_{k}$. Then 
\begin{align*}
    \tr B_{k+1}D^{2}\eta_{k} 
    &= \tr B_{k+1}D^{2} u_{k+1 } - \tr B_{k+1}D^{2} u_{k} \\
    &= nf^{1/n} - \tr (\det D^{2}u_{k})^{1/n}\,(D^{2}u_{k})^{-1} D^{2} u_{k} \\
    &= nf^{1/n} - n (\det D^2 u_k)^{1/n} = nf^{1/n} - n(\nu_k)^{1/n}.
\end{align*}
Hence $\eta_k$ solves the following linear elliptic PDE, the coefficients of which are elements of the matrix $B_{k+1}\coloneqq \cof D^{2}u_{k}$,
\begin{equation*}
    \begin{cases}
        \tr B_{k+1}D^{2}\eta_{k}=n(f)^{1/n}-n(\nu_{k})^{1/n} \text{ in } \Omega,\\
        \eta_{k}=0 \text{ on } \partial \Omega.
    \end{cases}
\end{equation*}
Since the operator $L_{k+1}(v)\coloneqq \tr B_{k+1}(D^{2}v)$ is strictly elliptic and Eq.~\eqref{eq: bounds on nu_k}) gives $n(f)^{1/n}-n(\nu_{k})^{1/n}\geq 0$, the maximum principle for linear elliptic PDEs (\cite[Theorem 3.1]{Gilbarg}) yields
\begin{equation*}
     u_{k}-u_{k-1} = \eta_{k}\leq \max_{\bar{\Omega}}\eta_{k}=\max_{\partial \Omega}\eta_{k}=0,
\end{equation*}
or in other words, for every $k \ge 0$.
\begin{equation}
    u_{k} \leq u_{k-1} \text{ on } \bar{\Omega}. \label{eq: decreasing seq}
\end{equation}
It follows from the above and from the comparison 
    principle for the real Monge--Ampère operator (see for example~\cite[Theorem~3.21]{Le}) that $\{u_{k}\}_{k\geq 0}$ is a decreasing sequence bounded from below by $u_{MA}$, the solution to the corresponding Monge--Ampère equation~\eqref{eq:MA}. Hence, the sequence converges to some function $\tilde{u}\coloneqq \lim_{k\to\infty}u_{k}$. Moreover, for every $k \geq 0$ and $\psi\in C^{\infty}_{0}(\Omega)$
\begin{equation*}
    \int_{\Omega} (\partial_{ij}^{2}u_{k+1}) \psi dV=\int_{\Omega}u_{k} (\partial_{ij}^{2}\psi) dV.
\end{equation*}
The monotone convergence theorem implies that all second partial derivatives $\{\partial_{ij}^{2}u_{k}\}_{k}$ converge weak-star. Hence, we can define a limit function $\nu\coloneqq \weakstar_{k\to \infty}\nu_{k}$. 
Since $0<\epsilon^{n}\leq \nu, \nu_{k}\leq f$ for all $k\geq 0$ (Eq. \eqref{eq: bounds on nu_k}), we can interpret $\nu_{k}$ and $\nu$ as finite Borel measures. We note that $\{u_{k}\}_{k\geq 0}$ are Aleksandrov solutions to the following non-linear Dirichlet problems:
\begin{equation} \label{eq: Monge Ampere seq}
    \begin{cases}
        \det D^{2}u_{k}(x)=\nu_{k}(x), & x \in \Omega, \\
        u_{k}(x)=\phi(x), & x\in  \partial \Omega, \\
        u_{k} \text{ strongly convex.}
    \end{cases}
\end{equation}
Next, we use compactness of Aleksandrov solutions to the Dirichlet problem \eqref{eq: Monge Ampere seq} (\cite[Theorem 3.35]{Le}) to conclude that the limit function $\tilde{u}\coloneqq \lim_{k\to\infty}u_{k}\in C(\bar{\Omega})$ and satisfies the following
\begin{equation*} 
    \begin{cases}
        \det D^{2}\tilde{u}(x)=\nu(x), & x \in \Omega, \\
        \tilde{u}(x)=\phi(x), & x\in  \partial \Omega, \\
        \tilde{u} \text{ convex.}
    \end{cases}
\end{equation*}
 Since the limit function is continuous and the sequence monotone, Dini's theorem gives that the convergence is uniform. 
It remains to show that $\tilde{u}$ solves the initial Monge--Ampère equation~\eqref{eq:MA} by showing that $f=\nu$ as measures. To this end, let us note the following
\begin{equation}  \label{eq: f equals to nu_k plus something}
\begin{split}
    n(f)^{1/n}&=\tr(B_{k+1}D^{2}u_{k+1})=\tr B_{k+1}(D^{2}u_{k+1}-D^{2}u_{k})+\tr B_{k+1}D^{2}u_{k}\\
    & =\underbrace{\tr B_{k+1}(D^{2}u_{k+1}-D^{2}u_{k})}_{\coloneqq \tau}+ n(\underbrace{\nu_{k}}_{\xrightarrow[]{w^{*}} \nu})^{1/n}.
\end{split}
\end{equation}
 Hence, it remains to show that $\tau\to 0$ weak-star. Since all second partial derivatives $\partial_{ij}^{2}u_{k}$ converge weak-star, it suffices to uniformly bound terms $(B_{k})_{ij}$. To this end, we first define bilinear forms $(x,y)_{k}\coloneqq (B_{k}x,y)$ 
for all $ x,y \in \mathbb{R}^{n}$. From Rayleigh quotients, we conclude that 
$\gamma^{k}_{min}\leq \|x\|^{2}_{k}/\|x\|^{2}\leq \gamma^{k}_{max}$ for all 
$x\in \mathbb{R}^{n}$ and for all $k\geq 0$. Let $\{e_{m}\}_{1\leq m\leq n}$ denote the standard basis in $\mathbb{R}^{n}$ and use the polarization identity to obtain
\begin{align*}
    (B_{k})_{ij}= (e_{j},e_{i})_{k}=(1/4)(\|e_{j}+e_{i}\|^{2}_{k}-\|e_{j}-e_{i}\|^{2}_{k}).
\end{align*}
Recalling Equation~\eqref{eq: bounds for eigenvalues of B_k}, it follows that
\begin{equation} \label{eq: bounds for B_k}
    \epsilon/\|f\|_{C(\bar{\Omega})}\leq \gamma_{min}^{k}\leq (B_{k})_{ij}\leq \gamma_{max}^{k}\leq \|f^{1/n}\|_{C(\bar{\Omega})}/\epsilon,
\end{equation}
for all $1\leq i,j\leq n$ and $ k\geq 0$.
Thus, we have that 
 \begin{equation*}
      \tau\coloneqq \tr B_{k+1}(D^{2}u_{k+1}-D^{2}u_{k}) \xrightarrow[]{w^{*}} 0.
 \end{equation*}
 By the above and Equation~\eqref{eq: f equals to nu_k plus something} we have
\begin{equation*}
    f-\nu=\weakstar_{k\to \infty} (f -\nu_{k})=(f^{1/n}-\nu_{k}^{1/n})\sum_{j=0}^{n-1}(-1)^{j}f^{(n-1-j)/n}\nu_{k}^{j/n}=0.
\end{equation*}
Thus, $f=\nu$ as measures. Since all second partial derivatives of $\{u_{k}\}_{k\geq 0}$ are bounded, the Arzel\`a--Ascoli theorem gives that in fact $\tilde{u}\in C^{1,1}(\bar{\Omega})$.
\end{proof}

\section{Description of our algorithm}\label{sec:algorithm}

In the remainder of the paper, we will restrict our attention to the two-dimensional case, $n=2$, but we expect the method to be competitive also in higher dimensions.

Theorem~\ref{thm:convergence} provides the idea for our numerical method, which we name \emph{Bellman's algorithm} and describe below.

First, we discretize the domain $\Omega$ using a rectangular $N \times N$ mesh. For simplicity of the implementation, we only consider the case where $\Omega$ is a 
bounded rectangle  in $\mathbb{R}^2$, despite the fact that these domains are not 
strictly convex, and  the theoretical framework for solving Monge--Ampère equations 
on rectangles is less well-developed. 

Next, we construct a sequence of functions $u_k$ where $u_k$ solves the discretized
version of the following Dirichlet problem, using a central difference second order finite 
difference solver:
\begin{equation}\label{eq: trace equation}
    \begin{cases}
        \tr(B_{k}(x)D^{2}u_{k}(x)) = 2\sqrt{f(x)}, & x \in \Omega, \\
        u_k(x) = \phi(x) ,      & x \in \partial\Omega.
    \end{cases}
\end{equation}
Here $B_k$ is a sequence of $\mathcal{B}$-valued functions constructed as follows:
\begin{enumerate}
    \item \label{step: Poisson_eq} $u_{0}$ is obtained by choosing $B_{0}=I$, i.e.\ the identity matrix. We observe that this renders equation \eqref{eq: trace equation} a Dirichlet problem for the Poisson equation.
    \item \label{step: B_matrices} for $k\geq 1$ we loop over the mesh and set $B_{k}\coloneqq \sqrt{\det D^{2}u_{k-1}}(D^{2}u_{k-1})^{-1}.$ If the computed $D^2u_{k-1}$ is not positive definite, we mark the corresponding 
    point $x$, and temporarily take $B_k(x)$ as the identity matrix. 
    
    \item\label{step:interpolation}
    We loop over the mesh a second time and replace $B_k(x)$ at the marked 
    points with the (normalized to determinant~1) convex combination of the computed
    Bellman matrices corresponding to the nearest  grid points in horizontal and 
    vertical direction for which $D^{2}u_{k-1}$ was positive definite.
\end{enumerate}
    
We terminate the computation when $\|u_k - u_{k-1}\|_{\infty} < 10^{-12}$ or when 
the number of iterations exceeds \num{10000}.

If the solution of the original Monge--Ampère equation is strictly convex, we 
expect the constructed functions $u_k$ to be strictly convex as well, at least when~$k$
sufficiently large, but it can happen that convexity fails in the first few steps -- see
Example~\ref{sec:example_2} -- even 
when the final solution is strictly convex. The interpolation
step~\eqref{step:interpolation} is necessary to assure convergence and in 
practice it pushes the approximants $u_k$  to become more and more convex as~$k$
increases. We will illustrate this phenomenon more clearly in Section~\ref{sec:examples}.

If the solution of the original Monge--Ampère equation is not strictly convex, i.e.\ 
if $f$ has zeros, we cannot expect our method to converge. In practice, our proposed 
algorithm still performs very well when the set of degeneracy is small, for example 
if $f$ has an isolated zero, or even if the zero-set of $f$ is one-dimensional 
(cf.~Example~\ref{ex:degenerate}). 

In the fully degenerate case, where $f$ vanishes on a non-empty open set or even vanishes identically, our proposed method fails to converge,
or at least converges very slowly (cf.~Example~\ref{ex:CircularDegeneracy}).
We intend to return to the degenerate Monge--Ampère equations in a subsequent paper.

In our implementation of the Bellman method, we use the Python package FinDiff~\cite{FINDIFF} to perform the necessary computations. Also,
for efficiency reasons, the computations in step~(\ref{step: B_matrices}) are 
performed in parallel, utilizing NumPy's tensor operations, thus removing the need 
for an explicit loop over the mesh elements.

For comparison, we also implemented two methods (M1 and M2) proposed by 
Benamou et al.~\cite{Benamou}.

The M2 method is a fixed-point method which in spirit is somewhat similar to our method. The idea is to solve a 
sequence of  Dirichlet problems for the Poisson equation, and update the right hand side in each iteration,  whereas we instead change the partial differential operator at  each step, keeping  the right hand side constant. 
For detailed descriptions of these methods, see~\cite{Benamou}, but in short, for M2 we solve
\begin{equation} \label{eq:M2}
    \begin{cases}
        \Delta u_k = g_k(x),   & x \in \Omega, \\
        u_k(x) = \phi(x),      & x \in \partial\Omega.
    \end{cases}
\end{equation}
where $g_0 = 2\sqrt{f}$ and $g_{k+1} = T(g_k) = \sqrt{(\Delta u_k)^2 + 2(f-\det(D^2 u_k))}$.

It's not too hard to show that the solution to the original Monge--Ampère equation~\eqref{eq:MA} is obtained when $g$ is a fixed point to the 
operator ~$T$.
Note that our proposed algorithm can also be viewed as a fixed-point method:
If~$u$ is a strictly convex solution of the Monge--Ampère equation~\eqref{eq:MA}, 
and $B(x)$ is the associated Bellman function, the solution
of~\eqref{eq: trace equation} reproduces $u$.

Method M1 is different. Using a central difference
discretization of the second order derivatives of $u$, the discretized 
version of the Monge--Ampère equation
\[
    (\mathcal{D}^2_{xx} u_{ij})(\mathcal{D}^2_{yy} u_{ij}) -
    (\mathcal{D}^2_{xy} u_{ij})^2 = f_{ij}
\]
is rewritten as a quadratic equation for $u_{ij}$ as
\begin{equation}\label{eq:M1}
    4(a_1 - u_{ij})(a_2 - u_{ij}) - \frac{1}{4}(a_3-a_4)^2 = h^4 f_{ij},
\end{equation}
where 
\begin{equation*}
    \begin{cases}
        a_1 = \frac{1}{2} ( u_{i+1,j} + u_{i-1,j} ), \\
        a_2 = \frac{1}{2} ( u_{i,j+1} + u_{i,j-1} ), \\
        a_3 = \frac{1}{2} ( u_{i+1,j+1} + u_{i-1,j-1} ), \\
        a_4 = \frac{1}{2} ( u_{i-1,j+1} + u_{i+1,j-1} ).
    \end{cases}
\end{equation*}
By solving~\eqref{eq:M1} for $u_{ij}$, selecting the smaller solution 
of the quadratic equation to impose local convexity, we have an
iterative scheme for M1:
\begin{equation}\label{eq:M1_iterative}
    u_{ij} = \frac{1}{2}\left( 
        a_1 + a_2 + \sqrt{(a_1-a_2)^2 + \tfrac{1}{4}(a_3-a_4)^2 + h^4 f_{ij}}
    \,\right),
\end{equation}
which is then solved using Gauss--Seidel iteration, keeping the boundary 
values fixed and updating one $u_{ij}$ at the time, 
using~\eqref{eq:M1_iterative}. We should note that this is computationally 
expensive, especially in our NumPy implementation where explicit loops 
are slow, so M1, and to some extent our proposed method, are the ones that
would benefit the most from an optimized implementation. Here, it would be natural to
update the whole matrix of functions values at once, using~\eqref{eq:M1_iterative}
instead of doing it element by element, but apparently this destroys the convergence 
to the actual solution of the Monge--Ampère equation.

In our comparison we chose the same termination criterion for all methods, 
i.e.\ to  terminate the computation when  $\|u_k - u_{k-1}\|_{\infty} < 10^{-12}$ or the number of iterations exceeds \num{10000} (for M1, 
we allow up to \num{300000} iterations, where one iteration means 
updating each $u_{ij}$ once).

\section{Numerical experiments}\label{sec:examples}

In this section, we show the results of testing the Bellman algorithm on 
several examples with various degrees of smoothness and regularity of the
right-hand side of the Monge--Ampère equation, $f$, on the square 
$\Omega = [-1,1]^{2}$, unless otherwise stated. Moreover, we reproduce 
the algorithms M1 and M2 presented in~\cite{Benamou}. 

Then we compare the performance of both algorithms in terms of the
approximation of the exact solution (restricted to the mesh) $\tilde{u}$ with the numerical 
solution $u$, i.e.\ $\|u-\tilde{u}\|_{\infty}$, $\|u-\tilde{u}\|_{2}$, 
the number of iterations and CPU time needed for convergence. We also
demonstrate the effectiveness of interpolating the Bellman functions, i.e.\ $B_{k}(x)$ from Equation \eqref{eq: trace equation},
in our new algorithm on  Example~\ref{sec:example_2}.

\subsection{Smooth, strictly convex examples}

\subsubsection{Standard example} 

The first example that we will look at has the exact solution
\begin{equation*} \label{eq: standard_example}
    u(x,y)= e^{\frac{1}{2}(x^{2} +y^{2})}
\end{equation*}
of the Monge--Ampère equation with a right-hand side given by
\begin{equation*}
    \MA(u) = f(x,y)= (1+x^{2} +y^{2})e^{x^{2} +y^{2}}.
\end{equation*}
This example has been used in several previous papers, see for 
example~\cite{Benamou, Froese2, Belgacem}. Clearly, $u$ is $C^\infty$-smooth and strictly convex, 
so we expect any reasonable solver to give good results for this example.

In fact, the Bellman method produces a sequence of strictly convex approximants, 
and the interpolation step is never invoked. Comparing the Bellman method to the 
other methods, all approaches converge to the same solution regardless of 
mesh size, see Figure~\ref{fig: StandardExample_diff}. The Bellman method needs 
6--7 iterations to converge, whereas the M2 method requires 40--50 iterations, independent of mesh size. Even if each iteration in our method is computationally a little more demanding
than in the M2 method, the overall running time is 5--6 times faster for the 
Bellman method, see Figure~\ref{fig: StandardExample_cpu}. Both methods have 
similar running time complexity ($\sim N^{2.8}$ and $\sim N^{2.9}$, respectively).

For the M1 method, each iteration is very fast, but the number of iterations
needed for convergence grows about linearly with $N$, and 
\num{1500}--\num{90000} iterations are needed which results in slow
total running times and a worse time complexity ($\sim N^{4.0}$). For a $255 \times 255$ grid, 
the M1 method ran in just under \qty{3}{h}, where the Bellman method 
took \qty{24}{s} and the~M2 method around \qty{2}{min}.

\begin{figure}[tbph]
     \centering
     \begin{subfigure}[t]{0.48\textwidth}
        % \begin{tikzpicture}
        % \begin{loglogaxis}[
        %     mesh plot style,
        %     legend pos=south west,
        % ]
        %  \addplot[ bellman style] table [x=N, y=sup(B), col sep=space] {StandardExample/table.txt};
        %  \addplot[mark=|, M1 style] table [x=N, y=sup(M1), col sep=space] {StandardExample/table.txt};
        %  \addplot[mark=square,M2 style] table [x=N, y=sup(M2), col sep=space]{StandardExample/table.txt};
        %  \addplot[mark=o, bellman style] table [x=N, y=l2(B), col sep=space] {StandardExample/table.txt};
        %  \addplot[mark=|, M1 style] table [x=N, y=l2(M1), col sep=space] {StandardExample/table.txt};
        %  \addplot[mark=square, M2 style] table [x=N, y=l2(M2), col sep=space] {StandardExample/table.txt};
        % \legend{Bellman,M1,M2}
        %  \end{loglogaxis}
        % \end{tikzpicture}
        \includegraphics{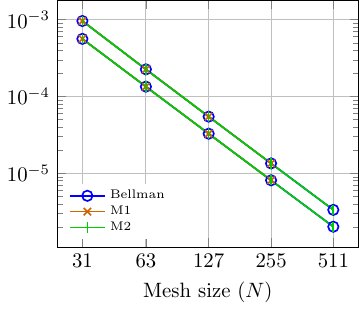}
         \caption{Error in supremum (top) and $\ell^{2}$ norm (bottom). The results for all three algorithms are almost identical and all graphs overlap. The results suggest that $\|u-u_N\| \sim N^{-2}$ for the sup-norm as well as the $\ell^2$-norm.}
         \label{fig: StandardExample_diff}
    \end{subfigure}
\hfill
\begin{subfigure}[t]{0.48\textwidth}
    % \begin{tikzpicture}[inner frame sep=0]
    % \begin{loglogaxis}[
    %     mesh plot style,
    %     legend pos=north west,
    %  ]
    %  \addplot[mark=o, bellman style] table [x=N, y=time(B), col sep=space] {StandardExample/table.txt};
    %  \addplot[mark=|, M1 style] table [x=N, y=time(M1), col sep=space] {StandardExample/table.txt};
    %  \addplot[mark=square, M2 style] table [x=N, y=time(M2), col sep=space] {StandardExample/table.txt};
    %  \legend{Bellman, M1, M2}
    %  \end{loglogaxis}
    % \end{tikzpicture}   
    \includegraphics{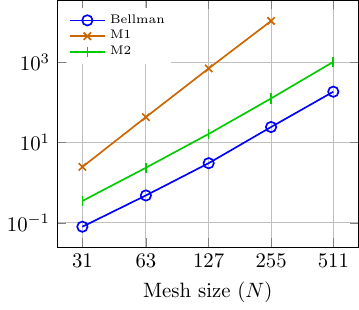}
     \caption{Running time (s). The results suggest that the running time is $\sim N^{2.8}$ for the Bellman method, $\sim N^{2.9}$ for M2 method and $\sim N^{4.0}$ for~M1.}
     \label{fig: StandardExample_cpu}
    \end{subfigure}
     \caption{Standard example: $u(x,y)= e^{0.5(x^2 +y^2)}$.}
     \label{fig: StandardExample_diff_and_diff1}
\end{figure}

\subsubsection{Degenerate example with regularisation}\label{sec:example_2}

\begin{figure}[tbph]
     \centering
     \begin{subfigure}[t]{0.48\textwidth}
     %    \begin{tikzpicture}
     %    \begin{loglogaxis}[
     %        mesh plot style,
     %        legend pos =south west
     %    ]
     %     \addplot[ bellman style] table [x=N, y=sup(B), col sep=space] {DegenerateExampleWithRegularisation/table.txt};
     %     \addplot[ M1 style] table [x=N, y=sup(M1), col sep=space] {DegenerateExampleWithRegularisation/table.txt};
     %     \addplot[ M2 style] table [x=N, y=sup(M2), col sep=space] {DegenerateExampleWithRegularisation/table.txt};
     %     \addplot[mark=triangle, blue!50!black, thick] table [x=N, y=sup(B2), col sep=space] {DegenerateExampleWithRegularisation/table2adj.txt};
     %     \addplot[mark=o, bellman style] table [x=N, y=l2(B), col sep=space] {DegenerateExampleWithRegularisation/table.txt};
     %     \addplot[ M1 style] table [x=N, y=l2(M1), col sep=space] {DegenerateExampleWithRegularisation/table.txt};
     %     \addplot[ M2 style] table [x=N, y=l2(M2), col sep=space] {DegenerateExampleWithRegularisation/table.txt};
     %     \addplot[mark=triangle, blue!50!black, thick] table [x=N, y=l2(B2), col sep=space] {DegenerateExampleWithRegularisation/table2adj.txt};
         
     % \legend{Bellman, M1, M2, Bellman2}
     %     \end{loglogaxis}
     %    \end{tikzpicture}
        \includegraphics{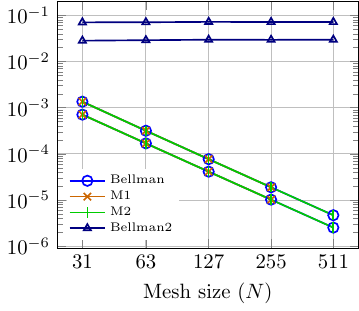}
         \caption{Error in supremum (top) and $\ell^{2}$ norm (bottom). The results for the 
         three algorithms are almost identical and suggest that $\|u-u_N\| \sim N^{-2}$, for both norms. The Bellman algorithm without interpolation (Bellman2 in the graph) does not converge.}
         \label{fig: DegenerateExampleWithRegularisation_diff}
    \end{subfigure}
 \hfill
\begin{subfigure}[t]{0.48\textwidth}
    % \begin{tikzpicture}[inner frame sep=0]
    % \begin{loglogaxis}[
    %     mesh plot style,
    %     legend pos=north west,
    %  ]
    %  \addplot[bellman style] table [x=N, y=time(B), col sep=space] {DegenerateExampleWithRegularisation/table.txt};
    %  \addplot[M1 style] table [x=N, y=time(M1), col sep=space] {DegenerateExampleWithRegularisation/table.txt};
    %  \addplot[M2 style] table [x=N, y=time(M2), col sep=space] {DegenerateExampleWithRegularisation/table.txt};
    %  \legend{Bellman, M1, M2}
    %  \end{loglogaxis}
    % \end{tikzpicture}    
    \includegraphics{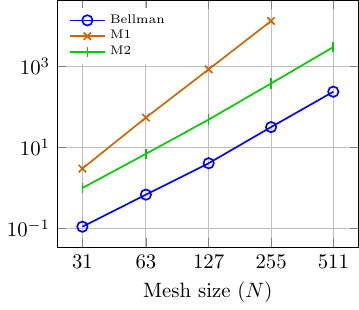}
     \caption{Running time (s). The results suggest that the running time is 
     $\sim N^{2.7}$ for the Bellman method, $\sim N^{2.8}$ for the M2 method and $\sim N^{4.0}$ for~M1.}
     \label{fig: DegenerateExampleWithRegularisation_cpu}
    \end{subfigure}
    \caption{Degenerate example with regularisation: $u(x,y)=0.5(x-0.5)^{4} + 0.1 x^{2}+y^{2}$.}
     \label{fig:example_2}
\end{figure}

\begin{figure}[tbph]
    \centering
    \begin{subfigure}[t]{0.24\textwidth}
    \centering
        \includegraphics[width=\textwidth]{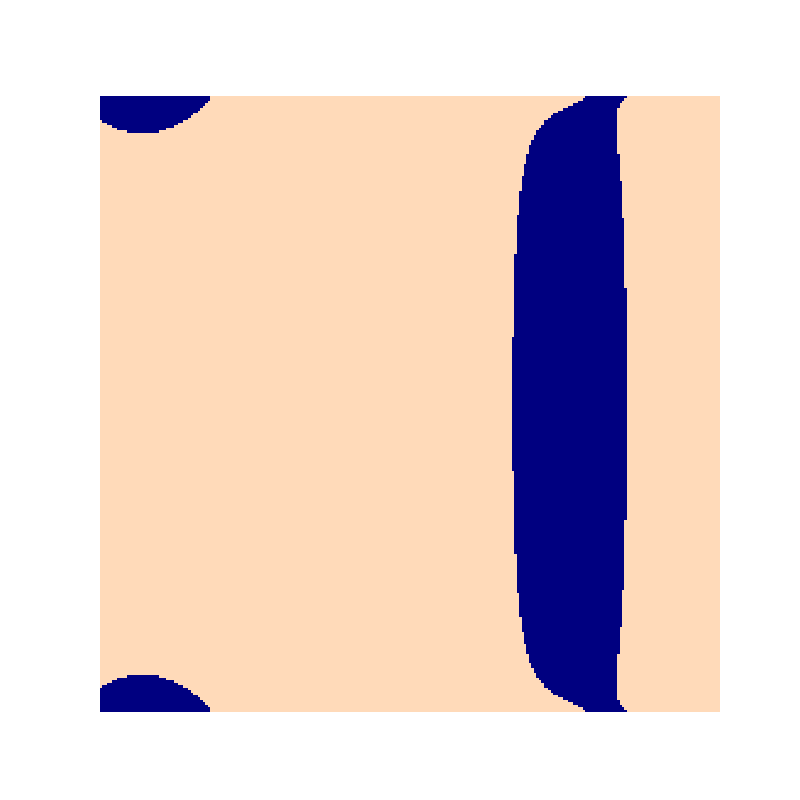}
        \caption{Iteration 1}
    \end{subfigure}
    \begin{subfigure}[t]{0.24\textwidth}
    \centering
        \includegraphics[width=\textwidth]{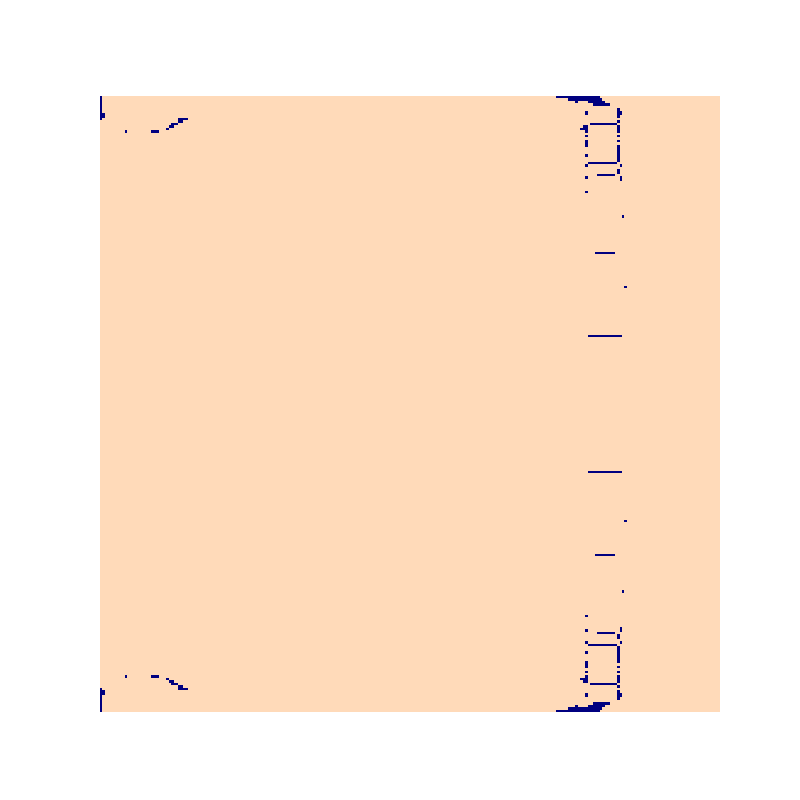}
        \caption{Iteration 2}
    \end{subfigure}
    \begin{subfigure}[t]{0.24\textwidth}
    \centering
        \includegraphics[width=\textwidth]{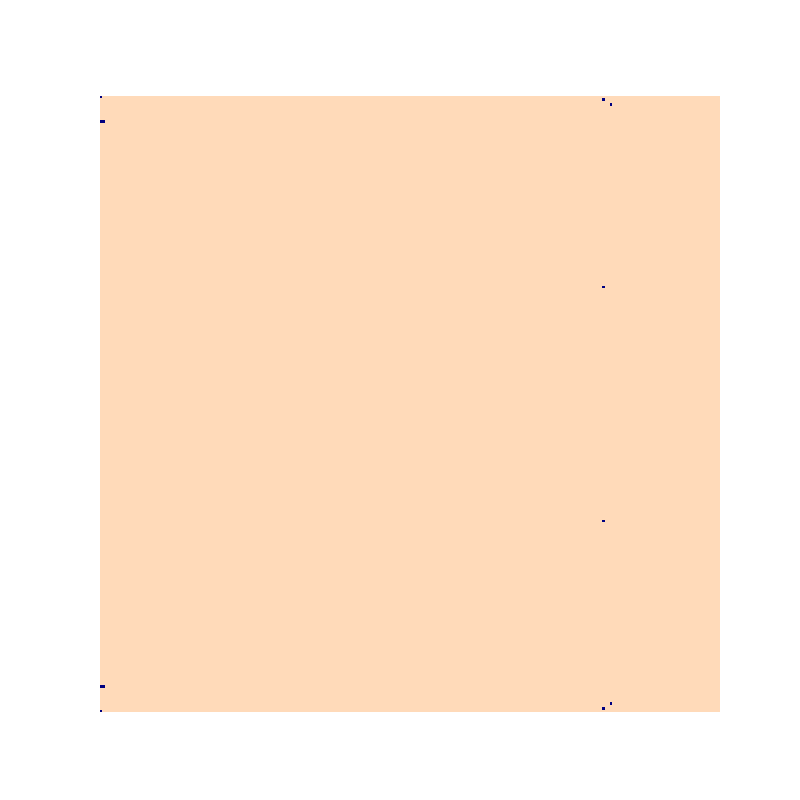}
        \caption{Iteration 3}
    \end{subfigure}
    \begin{subfigure}[t]{0.24\textwidth}
    \centering
        \includegraphics[width=\textwidth]{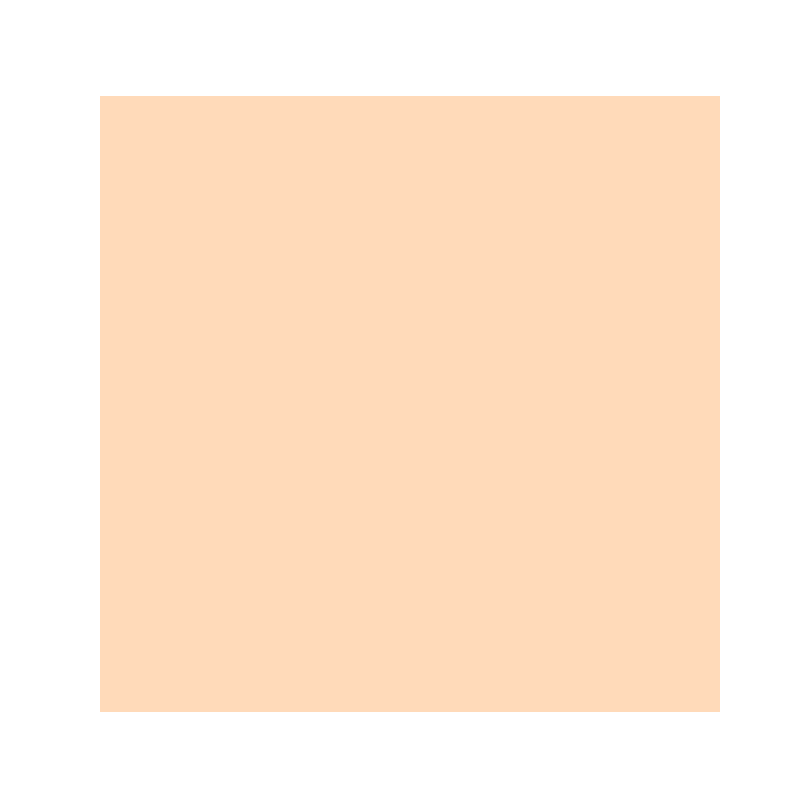}
        \caption{Iteration 4}
    \end{subfigure}
    \caption{Degenerate example with regularisation: $u(x,y)=0.5(x-0.5)^{4} + 0.1 x^{2}+y^{2}$. \\ The Bellman algorithm converges in 9 iterations (mesh size $N= 255$). The dark areas in the picture show the region where 
    the computed Hessian $D^2 u_k$ fails to be positive definite. After four
    iterations, the approximant is everywhere strictly convex and the interpolation step is no longer required.}
    \label{fig:example_2_singularity}
\end{figure}

For this example, the exact solution to the Monge--Ampère equation is  
\begin{equation}\label{eq: degenerate_example_with_regularisation}
    u(x,y)=0.5(x-0.5)^{4} + \epsilon x^{2}+y^{2},
\end{equation}
where $\epsilon \in \mathbb{R}_{+}$,
and the right-hand side of the equation is
\begin{equation*}\
    \MA(u) = f(x,y)= 12(x-0.5)^{2}+4\epsilon.
\end{equation*}
We note that the solution is smooth and strictly convex for each $\epsilon > 0$, 
but the smaller the choice of $\epsilon$, the less convex the exact solution is along the line $x=0.5$. 
We solve the Monge--Ampère equation for $\epsilon=0.1$ and observe that all three algorithms 
solve the Monge--Ampère equation successfully. 

However, the Bellman algorithm requires the additional interpolation step 
(step~\ref{step:interpolation}) in order to converge, see 
Figure~\ref{fig: DegenerateExampleWithRegularisation_diff}. In fact,  the solution to the 
corresponding Poisson equation,  $\Delta u = 2\sqrt{f}$ fails to be convex when 
$0 \le \epsilon \lesssim 0.2$, thus the first approximant produced by the Bellman method 
(as well as~M2) is not convex. After a few iterations in the Bellman method, 
convexity of  $u_k$ is achieved and the interpolation step is no longer needed, see Figure \ref{fig:example_2_singularity}.

For this example, the difference in running time is more pronounced, and the 
Bellman method runs \num{10}--\num{13} times faster than the M2 method. 
The Bellman method converges in less than~\num{10} iterations, whereas the M2 method require around~\num{140} iterations. See Figure~\ref{fig:example_2} for comparison. Again, the M1 method runs much slower and requires a large, 
increasing number of iterations as $N$ grows. For a $255 \times 255$ grid,
the Bellman method took \qty{31}{s}, the M2 needed
\qty{376}{s}, while M1 required around \num{3.5}{h}.

\subsection{Mildly singular examples}

\subsubsection{Trigonometric}

On the domain $\Omega=[0,1]^{2}$ we consider the example where the exact solution
is given by
\begin{equation*}
    u(x,y)=-\cos{(0.5\pi x)}-\cos{(0.5\pi y)}
\end{equation*}
with the corresponding right hand side of the Monge--Ampère equation
\begin{equation*}
    \MA(u) = f(x,y)=(0.5\pi)^{4}\cos{(0.5\pi x)}\cos{(0.5\pi y)}.
\end{equation*} 
We remark that $f=0$ on the boundary where $x=1$ or $y=1$. This implies that for 
$x\in \Omega$ near this part of the boundary, the Bellman function gives 
less and less strictly positive matrices $B_{k}(x)$. Despite this complication, 
all three algorithms converge and produce the same result, see Figure~\ref{fig: Trigonometric_diff.}. 
For this example, the Bellman algorithm runs 3~times faster than the M2 method, 
see Figure~\ref{fig: Triginometric_cpu}. The number of iterations needed for the convergence 
is 6--10 for the Bellman algorithm and~35 for the M2 algorithm, regardless of the mesh size. Yet again, M1 runs much slower: on a $255 \times 255$ grid, the Bellman method needs~\qty{31}{s}, the M2 method just over~\qty{90}{s} and M1 needs a little over~\qty{3}{h}.

\begin{figure}[tbph]
     \centering
         \begin{subfigure}[t]{0.48\textwidth}
        % \begin{tikzpicture}
        % \begin{loglogaxis}[
        %     mesh plot style,
        %     legend pos=south west,
        % ]
        %  \addplot[bellman style] table [x=N, y=sup(B), col sep=space] {Trigonometric/table.txt};
        %  \addplot[M1 style] table [x=N, y=sup(M1), col sep=space] {Trigonometric/table.txt};
        %  \addplot[M2 style] table [x=N, y=sup(M2), col sep=space] {Trigonometric/table.txt};
        %  \addplot[bellman style] table [x=N, y=l2(B), col sep=space] {Trigonometric/table.txt};
        %  \addplot[M1 style] table [x=N, y=l2(M1), col sep=space] {Trigonometric/table.txt};
        %  \addplot[M2 style] table [x=N, y=l2(M2), col sep=space] {Trigonometric/table.txt};
        %  \legend{Bellman, M1, M2} 
        %  \end{loglogaxis}
        % \end{tikzpicture}
        \includegraphics{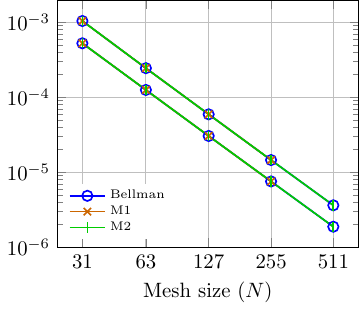}
         \caption{Error in supremum (top) and $\ell^{2}$ norm (bottom). The 
         results for all three algorithms are almost identical and suggest
         that $\|u-u_N\| \sim N^{-2}$ for both norms. All three graphs 
         overlap in the picture.}
         \label{fig: Trigonometric_diff.}
     \end{subfigure}
     \hfill
     \begin{subfigure}[t]{0.48\textwidth}
    % \begin{tikzpicture}[inner frame sep=0]
    % \begin{loglogaxis}[
    %     mesh plot style,
    %     legend pos=north west,
    %  ]
    %  \addplot[bellman style] table [x=N, y=time(B), col sep=space] {Trigonometric/table.txt};
    %  \addplot[M1 style] table [x=N, y=time(M1), col sep=space] {Trigonometric/table.txt};
    %  \addplot[M2 style] table [x=N, y=time(M2), col sep=space] {Trigonometric/table.txt};
    %  \legend{Bellman,M1, M2}
    %  \end{loglogaxis}
    % \end{tikzpicture} 
    \includegraphics{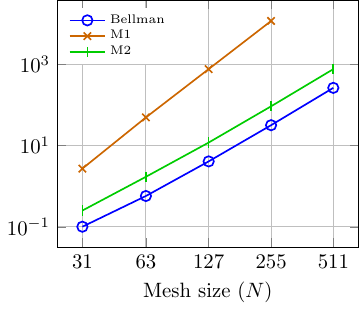}
         \caption{Running time (s). The results suggest that the running time 
         is $\sim N^{2.7}$ for the Bellman method and $\sim N^{2.8}$ for the M2, while the running time for M1 is $\sim N^{4.0}$.}
         \label{fig: Triginometric_cpu}
     \end{subfigure}
     \caption{Trigonometric: $u(x,y)=-\cos{(0.5\pi x)}-\cos{(0.5\pi y)}$.}
     \label{fig: Triginometric_cpu_and_diff}
\end{figure}

\subsubsection{Degenerate example}\label{ex:degenerate}

Let us reconsider Example~\ref{sec:example_2}, i.e. Equation~\eqref{eq: degenerate_example_with_regularisation}, and remove the 
regularisation term by setting $\epsilon=0$. In other words, we consider 
the exact solution
\begin{equation*}
    u(x,y)=0.5(x-0.5)^{4} + y^{2},
\end{equation*}
where 
\begin{equation*}
    \MA(u) = f(x,y) = 12(x-0.5)^{2}.
\end{equation*}
Note that the right-hand side vanishes along the line $x=0.5$, and while 
we can't expect the Bellman method to produce approximants with positive 
definite Hessians along this line, nevertheless, all three algorithms successfully 
tackle the degeneracy and converge to the correct solution, again producing the same results.

As in the strictly convex examples, the Bellman method requires around 10 iterations 
to converge, regardless of mesh size, but for this example the M2 method struggles, and the number of iterations grows, at a roughly linear rate, as the 
mesh size $N$ increases, see Table~\ref{table:degenerate_example}. This phenomenon 
heavily penalises the running time, which in the strictly convex settings 
was observed to be roughly proportional to $N^\alpha$ with $\alpha \approx 2.7$
for both the Bellman and M2 methods.

For this example, the running time of the Bellman method is still around $N^{2.7}$, 
but for the M2 method, the running time increases to $\sim N^{3.5}$ giving 
the Bellman method a huge advantage for large mesh sizes. For a $511 \times 511$ mesh,
the Bellman method ran in just over four minutes, and the method M2 needed 
almost eleven hours (on a reasonably fast personal computer). Here, the
performance of the~M1 method was closer to the~M2 method, but still about three times
slower.

\begin{table}[tbph]
    \caption{Running time for the degenerate example $u = 0.5(x-0.5)^4 + y^2$ on an $N \times N$ mesh. We only ran M1 on grid sizes up to $255 \times 255$.}\label{table:degenerate_example}
    \begin{tabular}{rrrrrrr}
        \toprule
        & \multicolumn{2}{c}{\emph{Bellman}} &  \multicolumn{2}{c}{\emph{M1}} &   \multicolumn{2}{c}{\emph{M2}}\\
         {$N$} & {iter}  & {time (s)} 
         & {iter} & {time (s)} & {iter} & {time (s)} \\
        \midrule 
  31  &  9   &  0.1   &        1990 &   3.0 &  260 &    2.0 \\
  63  &  9   &  0.7   &        7947 &  54.4 &  452 &   23.7 \\
 127  &  11  &   5.7  & \num{30559} & 867.7 &  758 &  269.6 \\
 255  &  10  &  33.8  & \num{115045}&  \num{13182.1}    & 1217 & 3327.4 \\
 511  &  10  & 243.1  &             &       & 1859 & \num{38946.9}  \\ 
 \bottomrule
    \end{tabular}
\end{table}

\begin{figure}[tbp]
     \centering
     \begin{subfigure}[t]{0.48\textwidth}
        % \begin{tikzpicture}
        % \begin{loglogaxis}[
        %     mesh plot style,
        %     legend pos=south west,
        %  ]
        %  \addplot[bellman style] table [x=N, y=sup(B), col sep=space] {DegenerateExample/table.txt};
        %  \addplot[M1 style] table [x=N, y=sup(M1), col sep=space] {DegenerateExample/table.txt};
        %  \addplot[M2 style] table [x=N, y=sup(M2), col sep=space] {DegenerateExample/table.txt};
        %  \addplot[ bellman style] table [x=N, y=l2(B), col sep=space] {DegenerateExample/table.txt};
        %  \addplot[M1 style] table [x=N, y=l2(M1), col sep=space] {DegenerateExample/table.txt};
        %  \addplot[M2 style] table [x=N, y=l2(M2), col sep=space] {DegenerateExample/table.txt};
        %  \legend{Bellman,M1, M2}
        %  \end{loglogaxis}
        % \end{tikzpicture}
        \includegraphics{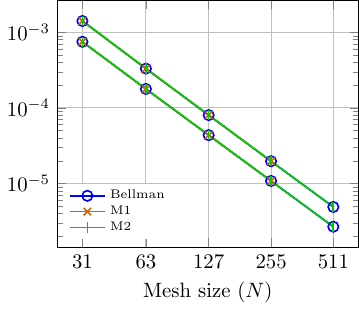}
         \caption{Error in supremum (top) and $\ell^{2}$ norm (bottom). 
         The results for all three algorithms are once again almost
         identical and suggest that $\|u-u_N\| \sim N^{-2}$ for both norms.}
         \label{fig: DegenerateExample_diff.}
    \end{subfigure}
 \hfill
\begin{subfigure}[t]{0.48\textwidth}
    % \begin{tikzpicture}[inner frame sep=0]
    % \begin{loglogaxis}[
    %     mesh plot style,
    %     legend pos=north west,
    %  ]
    %  \addplot[bellman style] table [x=N, y=time(B), col sep=space] {DegenerateExample/table.txt};
    %  \addplot[ M1 style] table [x=N, y=time(M1), col sep=space] {DegenerateExample/table.txt};
    %  \addplot[M2 style] table [x=N, y=time(M2), col sep=space] {DegenerateExample/table.txt};
    %  \legend{Bellman, M1, M2}
    %  \end{loglogaxis}
    % \end{tikzpicture}
    \includegraphics{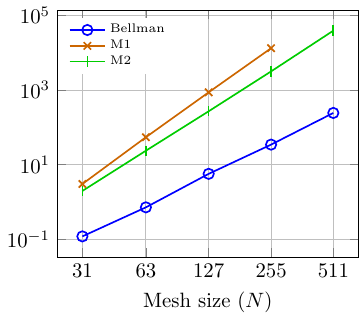}
     \caption{Running time (s). The results suggest that the running time is $\sim N^{2.7}$ for the Bellman method, $\sim N^{3.5}$ for M2 method and $\sim N^{4.0}$ for M1.}
     \label{fig: DegenerateExample_cpu}
    \end{subfigure}
     \caption{Degenerate example: $u(x,y)=0.5(x-0.5)^4 + y^2$.}
     \label{fig: DegenerateExample_cpu_and_diff}
\end{figure}

\subsubsection{Constant Monge--Ampère}\label{example:constant_MA}

Let us consider the Dirichlet problem 
\begin{equation}\label{eq:constant_MA}
    \begin{cases}
        \MA(u) \equiv 1 & \text{on $\Omega = [-1,1]^2$} \\
        u \equiv 1 & \text{on $\partial\Omega$}
    \end{cases}
\end{equation}
The exact solution of~\eqref{eq:constant_MA} is unknown, and even though the setup looks 
innocent enough at first 
glance, this example has shown to be a challenge to most if not all numerical algorithms 
that have been proposed for the Monge--Ampère equation. The problem is that the boundary 
conditions force $u$ to be almost constant on lines parallel to and close to the boundary of the square,
and the condition $\det(D^2 u) = 1$ then forces the second order derivative in the normal 
direction to blow up. Consequently, the eigenvalues of the Hessian are neither bounded from above nor away from zero on $\bar{\Omega}$. It follows from results by B\l{}ocki~\cite{Blocki} that the solution 
is $C^{1,1}(\Omega) \cap C(\bar\Omega)$, but by the argument above, we can't expect higher
order regularity of $u$ up to $\partial\Omega$.

\begin{figure}[tbph] 
    \centering
    % \begin{tikzpicture}[spy using outlines={rectangle, red, magnification=6, size=2.5cm, connect spies}]
    %     \node{ \includegraphics[width=0.98\textwidth]{diagonals_constant_MA.png}};
    %     \spy [red] on (-4.2,2.28) in node [left] at (1.5,1);
    % \end{tikzpicture}
    \includegraphics{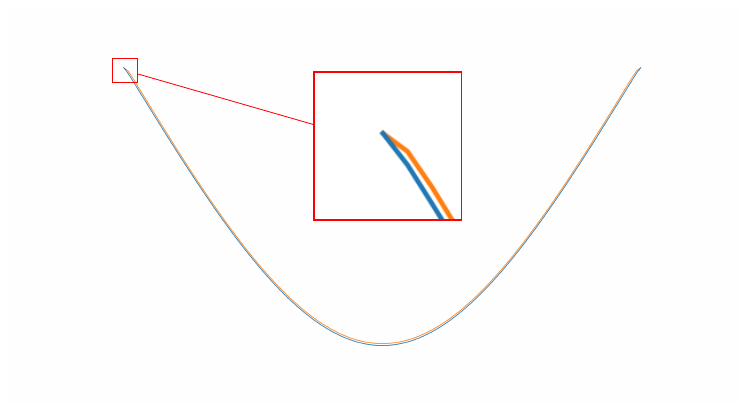}
    \caption{Constant MA example. The figure show cross-sections along the
    domain's diagonal of the solutions given by the Bellman and M1 algorithm, $u_{B}$ (top) and $u_{M1}$ (bottom), respectively. Both solutions 
    fail to be convex near the corners of the square}
    \label{fig: Constant_MA_solutions_diagonal}
\end{figure}

\begin{table}[tbph]
    \caption{Minimum value for the computed solution to the constant Monge--Ampère example. Method M1 and M2 produce the same values (to four points of accuracy). The computations with the M2 method were terminated after \num{10000} iterations, and the computations with M1 was terminated 
    after \num{300000} iterations, marked by * in the table. 
    The results for the wide stencil
    finite difference discretizations (4th--6th column) are taken from Oberman~\cite{Oberman}.}\label{table:constant_MA_min}
    \begin{tabular}{cccccc}
        \toprule
         {$N$}& {Bellman} &  { M1, M2} & {9pt stencil}& {17pt stencil} & {33pt stencil} \\
        \midrule 
      21 & 0.2917 & 0.2892 & 0.3115 & 0.2815 & 0.2839 \\
      41 & 0.2772 & 0.2734 & 0.3090 & 0.2807 & 0.2732 \\ 
      61 & 0.2721 & 0.2682 & 0.3082 & 0.2803 & 0.2711 \\
      81 & 0.2712 & 0.2655 & 0.3078 & 0.2802 & 0.2704 \\
     101 & 0.2694 & 0.2639 & 0.3076 & 0.2800 & 0.2700 \\
     127 & 0.2677 & 0.2626 \\
     255 & 0.2651 &  * \\
     511 & 0.2628 &  * \\
 \bottomrule
\end{tabular}
\end{table}

\begin{table}[tbp]
    \caption{Running time for the constant Monge--Ampère example. 
    Computations marked by *  were terminated after \num{10000} (for method M2) and 
    \num{300000} (for M1) iterations.}\label{table:constant_MA}
    \begin{tabular}{SSSSSSS}
        \toprule
        & \multicolumn{2}{c}{\emph{Bellman}} & \multicolumn{2}{c}{\emph{M1}} & \multicolumn{2}{c}{\emph{M2}} \\
         {$N$} & {iter} & {time (s)} & {iter} & {time (s)}  & {iter} & {time (s)} \\
        \midrule 
  31  &  17   &  0.1     &   2902  &      4.4 &   438    &      6.7  \\
  41  &  19   &  0.5     &   5086  &     14.3 &   1539   &     23.2  \\
  61  &  19   &  1.6     &  11146  &     69.9 &   3224   &    150.3  \\
  81  &  21   &  3.8     &  19379  &    222.5 &   5443   &    546.8  \\
 101  &  19   &  5.9     &  29698  &    518.3 &   8158   &   1513.7  \\
 127  &  20   &  10.8    &  46152  &   1297.7 &    *     &   *       \\
 255  &  29   & 102.7    & 174401  &  20896.0 &    *     &   *       \\
 511  &  25   & 870.7    &   *     &   *      &    *     &   *       \\
 \bottomrule
\end{tabular}
\end{table}

Running the Bellman algorithm on Equation~\eqref{eq:constant_MA} produces a function that 
(barely) fails to be convex near the corners of $\Omega$, but the same is true for 
Method M2, see Figure~\ref{fig: Constant_MA_solutions_diagonal}.
This example was also studied by Dean and Glowinski~\cite{Dean} 
(using a variational formulation) and Oberman~\cite{Oberman} 
(wide stencil finite difference discretization), all with similar results. 
In~\cite{Benamou}, they suggest comparing the different methods using 
the minimum (i.e.\ the value at the origin) of the solution. The wide stencils 
methods from Oberman~\cite{Oberman} are known to converge monotonically to a viscocity supersolution 
of Equation~\eqref{eq:MA}, and are thus too large. If the solution produced by the
Bellman method had been convex, monotonicity  would have implied 
that it too is too large, but since convexity fails near the corners of the square, 
we have no guarantee for this. Our method produces a solution which is slightly 
larger than the M2 and M1 methods, see Table~\ref{table:constant_MA_min}.

The two methods used in~\cite{Benamou} were observed to have running times 
of $\sim N^{4.0}$ (M1) and $\sim N^{4.6}$ (M2), respectively (results were reported 
for grid sizes $21 \le N \le 141$). When we ran those two methods on grid sizes
up to $255 \times 255$ (for M1), we observed similar running times.

The M2 method converges very slowly even for moderately large grid sizes,
and was terminated after 10,000 iterations when $N \ge 127$. The available results
for smaller meshes suggest a running time of $\sim N^{4.6}$. For this example, M1 outperforms M2 and runs at $\sim N^{4.0}$. The Bellman method is an order of magnitude faster 
still, and clocks in at $\sim N^{3.0}$, see Table~\ref{table:constant_MA}.
Already on a $101 \times 101$ grid, the Bellman method runs more than~300 times 
faster than the M2 method and about~80 times faster than~M1.

\begin{figure}[htb]
     \centering
     \begin{subfigure}[t]{0.48\textwidth}
        % \begin{tikzpicture}
        % \begin{loglogaxis}[
        %     mesh plot style,
        %     ymin=1e-6,
        %    legend pos=south west,
        % ]
        %  \addplot[bellman style] table [x=N, y=sup(B), col sep=space] {CircularDegeneracy/table.txt};
        %  \addplot[M1 style] table [x=N, y=sup(M1), col sep=space] {CircularDegeneracy/table.txt};
        %  \addplot[M2 style] table [x=N, y=sup(M2), col sep=space] {CircularDegeneracy/table.txt};
        %  \addplot[ bellman style] table [x=N, y=l2(B), col sep=space] {CircularDegeneracy/table.txt};
        %  \addplot[ M1 style] table [x=N, y=l2(M1), col sep=space] {CircularDegeneracy/table.txt};
        %  \addplot[ M2 style] table [x=N, y=l2(M2), col sep=space] {CircularDegeneracy/table.txt};
        %  \legend{Bellman, M1, M2}
        %  \end{loglogaxis}
        % \end{tikzpicture}
        \includegraphics{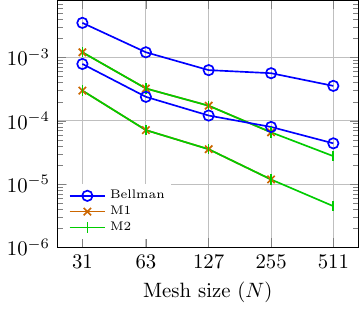}
         \caption{Error in supremum and $\ell^{2}$ norm. The errors for Bellman method are worse with $\|u-u_N\|_{\infty} \sim N^{-0.8}$ and $\|u-u_N\|_{\ell^2} \sim N^{-1.0}$. The other methods produce 
         almost identical results with $\|u-u_N\|_{\infty} \sim N^{-1.3}$ and $\|u-u_N\|_{\ell^2} \sim N^{-1.5}$.}
         \label{fig: CircularDegeneracy_diff.}
     \end{subfigure}
     \hfill
     \begin{subfigure}[t]{0.48\textwidth}
    % \begin{tikzpicture}[inner frame sep=0]
    % \begin{loglogaxis}[
    %     mesh plot style,
    %     legend pos=north west,
    %  ]
    %  \addplot[bellman style] table [x=N, y=time(B), col sep=space] {CircularDegeneracy/table.txt};
    %  \addplot[M1 style] table [x=N, y=time(M1), col sep=space] {CircularDegeneracy/table.txt};
    %  \addplot[M2 style] table [x=N, y=time(M2), col sep=space] {CircularDegeneracy/table.txt};
    %  \legend{Bellman, M1, M2}
    %  \end{loglogaxis}
    % \end{tikzpicture}  
    \includegraphics{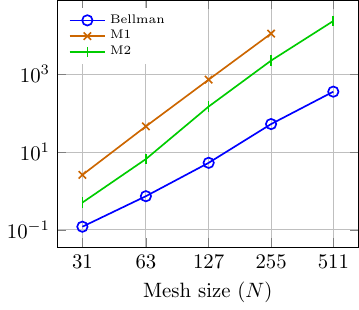}
         \caption{Running time (s). The results suggest that the running time is $\sim N^{2.9}$ for the Bellman method, $\sim N^{3.9}$ for the M2 method and $~\sim N^{4.0}$ for~M1.}
         \label{fig: CircularDegeneracy_cpu}
     \end{subfigure}
     \medskip 
     
     \begin{subfigure}[t]{0.45\textwidth}
         \centering
         \includegraphics[width=1\textwidth,height=0.92\textwidth]{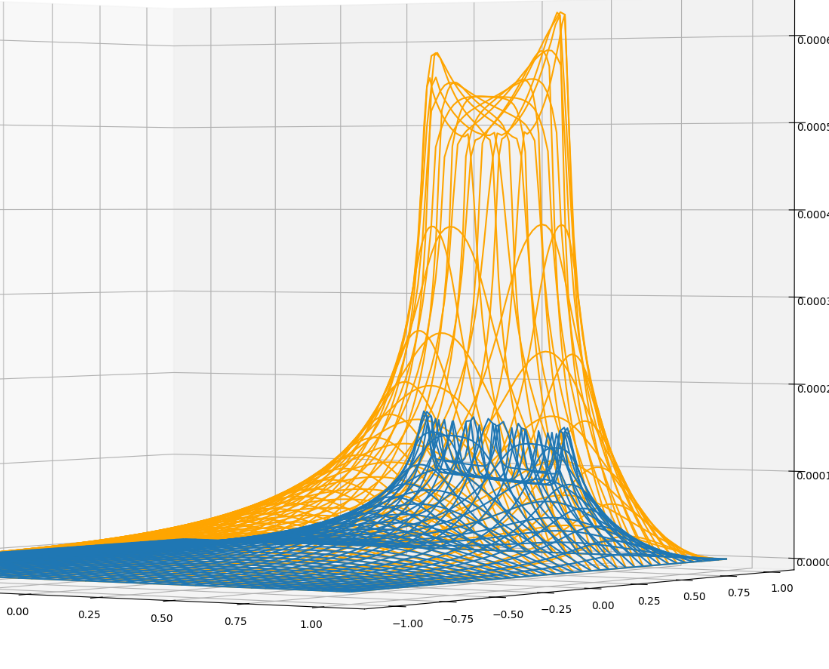}
         \caption{Difference $u_{B}-u$ (light) and $u_{M2}-u$ (dark), given by Bellman and Method M2, respectively.}
         \label{fig: Circular_Degeneracy_diff_solutions}
     \end{subfigure}
     \hfill
    \begin{subfigure}[t]{0.48\textwidth}
        % \begin{tikzpicture}[inner frame sep=0]
        % \begin{loglogaxis}[
        %     mesh plot style 2,
        %     legend pos=north east,
        %  ]
        %  \addplot[bellman style] table [x=time(B), y=l2(B), col sep=space] {CircularDegeneracy/table.txt};
        %  \addplot[M1 style] table [x=time(M1), y=l2(M1), col sep=space] {CircularDegeneracy/table.txt};
        %  \addplot[M2 style] table [x=time(M2), y=l2(M2), col sep=space] {CircularDegeneracy/table.txt};
        %  \legend{Bellman, M1, M2}
        %  \end{loglogaxis}
        % \end{tikzpicture} 
        \includegraphics{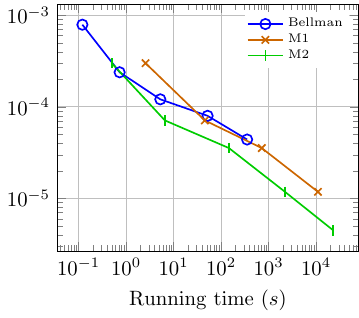}
             \caption{Comparing the running time to the achieved accuracy ($\ell^2$ norm),
             the performance of the three methods are fairly similar.}\label{fig:Circular_degeneracy_acc_to_cpu}
    \end{subfigure}
     \caption{Circular Degeneracy:\\ $u(x,y)= \frac{1}{2}(\sqrt{(x-0.5)^2+(y-0.5)^2}-0.2)^{+}$.}
     \label{fig: Circular_degeneracy_diff_cpu_diff_solutions}
\end{figure}

\subsection{Singular examples}

\subsubsection{Circular degeneracy}\label{ex:CircularDegeneracy}

We consider 
\begin{equation}\label{eq: circular_degeneracy}
   u(x,y)= 0.5(\sqrt{(x-0.5)^2+(y-0.5)^2}-0.2)^{+},
\end{equation}
which would be a solution to the Monge--Ampère equation with the following right-hand side
\begin{equation*}
    \MA(u) = f(x,y)=\frac{(\sqrt{(x-0.5)^2+(y-0.5)^2}-0.2)^{+}}{\sqrt{(x-0.5)^2+(y-0.5)^2}}.
\end{equation*}
In this example, $u \in C^{1,1}$, but $f = \MA(u)$ vanishes on a disc 
centred at $(0.5,0.5)$ with radius $r=0.2$. 
Here, both the~M1 and the~M2 methods outperform the Bellman algorithm, see Figure~\ref{fig: CircularDegeneracy_diff.}. As expected, the Bellman algorithm struggles on large areas where the exact solution is not strictly convex, that is on the aforementioned disc. This is the area where the solutions $u_{\text{B}}$ and $u_{M2}$ deviate the most from the exact solution $u$, see Figure~\ref{fig: Circular_Degeneracy_diff_solutions}. In fact, we see that the error in the supremum and $\ell^{2}$ norms are proportional to~$N^{-0.8}$ and~$N^{-1.0}$, respectively, for the Bellman algorithm, whereas they are around $N^{-1.3}$ and $N^{-1.5}$ for the other methods.

While the accuracy of the 
Bellman method is notably worse than the other methods, it seems like all methods still converge to the exact solution as 
the mesh size $N$ tends to $\infty$, but the rate of convergence, for all methods, is 
slower than for the strictly convex examples, see Figure~\ref{fig: CircularDegeneracy_diff.}. 

Even though the Bellman method gives worse results, it still outperforms the other methods
with respect to running time. The measured time complexity for the Bellman method is similar 
to the strictly convex examples ($\sim N^{2.9}$), whereas the running time of the M2 method 
is~$\sim N^{3.9}$ , resulting in a running time that is about~70 times slower than the Bellman method 
for a $511 \times 511$ mesh. The number of iterations required is 6--14 for the Bellman method, 
where Method M2 needs an increasing number of iterations as the number of mesh points increases
(1092 for the $511 \times 511$ mesh). The M1 method runs even slower. 
If we look at accuracy as a function of computational time (see Figure~\ref{fig:Circular_degeneracy_acc_to_cpu}), we see that all three methods
perform fairly similarly.

\begin{remark}
    We would like to remark that the Bellman algorithm fails to converge when the right-hand side vanishes on the domain, i.e. $f\equiv 0$ on $\Omega$. One example of such a function could be $u(x,y)=|x|$. Here, the solution given by our method is the solution to the Poisson equation that is the first iteration. Method M2 converges. However, the convergence is very slow, and the running time grows from 1~min to 6~hours as the mesh size increases, and with diminished accuracy, as the method is terminated after 10\,000 iterations.
\end{remark}

\subsubsection{Unbounded Monge--Ampère}

We investigate the convergence for an unbounded right-hand side.
For $\Omega=[0,1]^{2}$ and
\begin{equation*}
    u(x,y)=-\sqrt{2-x^2-y^2}.
\end{equation*}
The right-hand side of the equation is 
\begin{equation*}
    \MA(u) = f(x,y)=\frac{2}{(2-x^2-y^2)^2}
\end{equation*}
and we note that $f$ is unbounded near the point $(1,1)$. 
All three algorithms struggle, and only achieve an accuracy of order $N^{-0.5}$
(measured in sup-norm), which is the same as observed by Benamou et al.~\cite{Benamou} for~M1 and~M2. 
Most of the error is 
concentrated near the problematic point $(1,1)$, and if convergence is measured by
average $\ell^2$ norm instead, the accuracy is $N^{-1.5}$. All algorithms run 
slowly, and the Bellman method and M2 runs at roughly 
the same speed, while M1 is about 20 times slower.

Curiously enough, the Bellman method seems to enjoy 
a relative speed-up as well as an improvement in 
accuracy for larger meshes, see Figure~\ref{fig: UnboundedMA_cpu_diff_solutions_delta}.

If we instead do the computations on a slightly smaller square, e.g.~$[0, 0.99]^2$,
the three methods run at full accuracy ($ \sim N^{-2}$), and the Bellman method runs 
2--3 times faster than M2.

\begin{figure}[h!]
     \centering
     \begin{subfigure}[t]{0.48\textwidth}
        % \begin{tikzpicture}
        % \begin{loglogaxis}[
        %     mesh plot style,
        %    legend pos=south west
        % ]
        %  \addplot[bellman style] table [x=N, y=sup(B), col sep=space] {UnboundedMA/table.txt};
        %  \addplot[M1 style] table [x=N, y=sup(M1), col sep=space] {UnboundedMA/table.txt};
        %  \addplot[M2 style] table [x=N, y=sup(M2), col sep=space] {UnboundedMA/table.txt};
        %  \addplot[bellman style] table [x=N, y=l2(B), col sep=space] {UnboundedMA/table.txt};
        %  \addplot[M1 style] table [x=N, y=l2(M1), col sep=space] {UnboundedMA/table.txt};
        %  \addplot[M2 style] table [x=N, y=l2(M2), col sep=space] {UnboundedMA/table.txt};
        %  \legend{Bellman, M1, M2}
        %  \end{loglogaxis}
        % \end{tikzpicture}
        \includegraphics{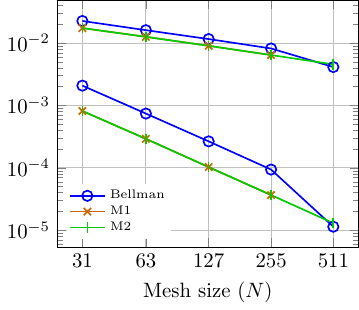}
         \caption{Error in supremum (top) and $\ell^{2}$ norm (bottom). The results for all three algorithms are fairly close, but not identical, and suggest that $\|u-u_N\|_{\infty} \sim N^{-0.5}$ and $\|u-u_N\|_{\ell^2} \sim N^{-1.5}$.}
         \label{fig: UnboundedMA_diff.}
     \end{subfigure}
     \hfill
     \begin{subfigure}[t]{0.48\textwidth}
    % \begin{tikzpicture}[inner frame sep=0]
    % \begin{loglogaxis}[
    %     mesh plot style,
    %     legend pos=north west,
    %  ]
    %  \addplot[bellman style] table [x=N, y=time(B), col sep=space] {UnboundedMA/table.txt};
    %  \addplot[ M1 style] table [x=N, y=time(M1), col sep=space] {UnboundedMA/table.txt};
    %  \addplot[ M2 style] table [x=N, y=time(M2), col sep=space] {UnboundedMA/table.txt};
    %  \legend{Bellman, M1, M2}
    %  \end{loglogaxis}
    %\end{tikzpicture}  
    \includegraphics{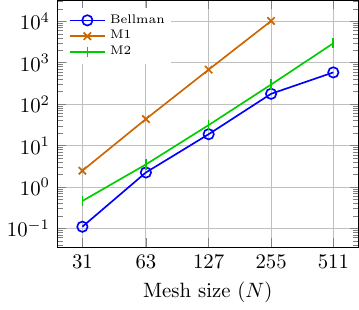}
         \caption{Running time (s). The results suggest that the running time is $\sim N^{3.0}$ for the Bellman method, $\sim N^{3.1}$ for M2 and $\sim N^{4.0}$ for~M1.}
         \label{fig: UnboundedMA_cpu}
     \end{subfigure}
    \begin{subfigure}[t]{0.48\textwidth}
        % \begin{tikzpicture}
        % \begin{loglogaxis}[
        %     mesh plot style 2,
        %    legend pos=south west
        % ]
        %  \addplot[bellman style] table [x=time(B), y=l2(B), col sep=space] {UnboundedMA/table.txt};
        %  \addplot[M1 style] table [x=time(M1), y=l2(M1), col sep=space] {UnboundedMA/table.txt};
        %  \addplot[M2 style] table [x=time(M2), y=l2(M2), col sep=space] {UnboundedMA/table.txt};
        %  \legend{Bellman, M1, M2}
        %  \end{loglogaxis}
        % \end{tikzpicture}
        \includegraphics{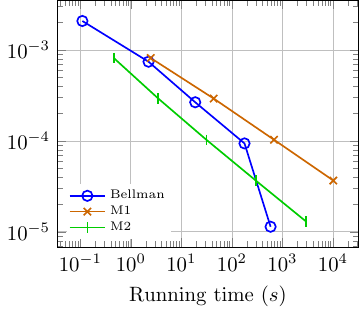}
         \caption{Comparing the running time to achieved accuracy ($\ell^2$
         norm), the performance of the three methods are fairly similar, with 
         the Bellman method winning for very fine meshes.}
     \end{subfigure}
     \caption{Unbounded MA: $u(x,y)=-\sqrt{2-x^2-y^2}$.}
     \label{fig: UnboundedMA_cpu_diff_solutions_delta}
\end{figure}

\section{Conclusions}
The Bellman algorithm relies heavily on the assumption that the solution to the Monge--Ampère 
equation is strictly convex on the domain. When this assumption is satisfied, the performance 
of the algorithm is very good. Using an $N\times N$ mesh, the discretization error in supremum 
and $\ell^{2}$ norm is proportional to $N^{-2}$, and is achieved with few iterations (typically 
fewer than $10$) resulting in a running time which is proportional to $N^{2.7}$--$N^{2.8}$. 

Similar performance was also observed in cases where the right hand side is mildly degenerate, 
such as the set where $f=0$ is at most one dimensional. In this setting, our method requires 
a few more iterations to converge (typically 10--15), but this only has a minor effect on the 
running time. In comparison, the M2 method converges a lot slower for this class of 
examples, typically needing more and more iterations as the number of grid points increases,
resulting in running times that vary between $N^{2.8}$ and $N^{3.6}$. 

The net effect is that our method runs 3--10 times faster than the M2 method on smooth
strictly convex examples, but as much as 20--100 (or more) times faster for the mildly degenerate 
cases. The M1 method is even slower, and has a running time proportional to around $N^{4.0}$ for most examples. With the exception of the constant Monge--Ampère example, M1 runs substantially slower than M2.

In the most challenging examples, when $f$ vanishes on an open set or when $f$ is unbounded,
the Bellman method begins to struggle. While the running time and the number of iterations is 
similar to the less challenging examples, the accuracy becomes worse resulting in supremum norm
error between $N^{-0.8}$ and $N^{-0.5}$. In these examples, the errors are mostly concentrated 
near the zero set (or the unbounded locus) of $f$, and measuring the error in $\ell^2$ norm gives 
slightly better results, with an error between $N^{-1.0}$ and $N^{-1.5}$.

On the other hand, other published methods struggle with these more challenging examples as well.
Method M2 outperforms the Bellman method in the circular degeneracy example
(Example~\ref{ex:CircularDegeneracy}), at the cost of very slow running times. In all the other examples
that we have considered, the accuracy of the Bellman method, M2 method and~M1 are very similar, 
indicating that all methods converge to the optimal discretization of the underlying PDE on the 
given mesh, and in almost all examples the Bellman method runs considerably faster than the remaining methods.

% \section{Declarations}
% \textbf{Conflict of interests and funding:} The authors have no relevant financial or non-financial interests to disclose.

\end{document}